\documentclass[graybox]{svmult}

\usepackage{mathptmx}       
\usepackage{helvet}         
\usepackage{courier}        
\usepackage{type1cm}        

\usepackage{makeidx}         
\usepackage{graphicx}        
\usepackage{multicol}        
\usepackage[bottom]{footmisc}

\makeindex             

\usepackage{amssymb}
\usepackage{latexsym}
\usepackage{amsfonts}
\usepackage{amsmath}

\usepackage{algorithm}
\usepackage{algorithmic}

\usepackage{subfigure}
\usepackage{subfigmat}
\usepackage{url}
\usepackage{hyperref}

\newcommand{\bxi}{\boldsymbol{\xi}}

\newcommand{\be}{\mathbf{e}}

\newcommand{\bi}{\mathbf{i}}
\newcommand{\bj}{\mathbf{j}}
\newcommand{\bk}{\mathbf{k}}

\newcommand{\bx}{\mathbf{x}}

\def\cI{\mathcal{I}}
\def\cO{\mathcal{O}}
\def\cA{\mathcal{A}}
\def\cR{\mathcal{R}}

\begin{document}

\title*{Local and Dimension Adaptive Sparse Grid Interpolation and Quadrature}
\author{J.D Jakeman and S.G. Roberts}
  
\institute{J.D. Jakeman, 
	  Department of Mathematics, Purdue University, West Lafayette, IN USA 47907, \texttt{jjakeman@purdue.edu}
	  \and
	  S.G. Roberts,
	  Mathematical Sciences Institute, Australian National University, ACT Australia 0200, \texttt{stephen.roberts@anu.edu.au}
}
\maketitle

\begin{abstract}
In this paper we present a locally and dimension-adaptive sparse grid method 
for interpolation and integration of high-dimensional functions with discontinuities. 
The proposed algorithm combines the strengths of the generalised sparse grid
algorithm and hierarchical surplus-guided local adaptivity. 
A high-degree basis is used to obtain a high-order method which, 
given sufficient smoothness, performs significantly better than the piecewise-linear basis.
The underlying generalised sparse grid algorithm greedily selects the dimensions and variable interactions that 
contribute most to the variability of a function. 
The hierarchical surplus of points within the sparse grid is used as an error criterion
for local refinement with the aim of concentrating computational effort within rapidly varying or discontinuous regions. 
This approach limits the number of points that are invested in `unimportant'
dimensions and regions within the high-dimensional domain.
We show the utility of the proposed method for non-smooth functions with hundreds of variables. 

\end{abstract}

\section{Introduction}
The need for interpolation and integration of high-dimensional functions arises in many fields including 
finance, physics, chemistry, and uncertainty quantification. 
Sparse grids have emerged as an extremely useful tool to construct such 
multi-dimensional approximations. They have been extensively used for
high-dimensional interpolation~\cite{barthelmann00,gerstner98} and quadrature~\cite{bungartz03,gerstner03}
 and have been shown, under certain conditions, to obtain 
significantly higher rates of convergence than many existing methods. For example, the complexity of the Monte Carlo method with $n$ samples is
$O(n^{-1/2})$, whereas the complexity of the sparse grid method~\cite{smolyak63} is $O(n^{-r}(\log n)^{(d-1)(r+1)})$ when used to approximate
$d$-dimensional integrands which have bounded mixed partial derivatives of order $r$. 
Sparse grids achieve faster rates of convergence by taking advantage of higher smoothness and lower-effective dimensionality of the integrand.

Standard sparse grids are isotropic, treating all dimensions equally. Although
an advance on alternative methods, such approximations can still be
improved. Many problems vary rapidly in only some dimensions, remaining 
less variable in other dimensions. Consequently it is advantageous to increase the
level of accuracy only in certain highly varying dimensions, resulting in so-called
adaptive or anisotropic grids. In some cases the important dimensions can be
determined a priori, but in most cases the grid points must be chosen
during the computational procedure. 

Gerstner and Griebel~\cite{gerstner03} developed a dimension-adaptive tensor-product quadrature method
to approximate high-dimensional functions by a sum of lower-dimensional terms. The method is based
upon a generalisation of the traditional isotropic sparse grid index set that, given appropriate error estimators,
can automatically concentrate computational effort in important dimensions. 
Recently Griebel and Hotlz~\cite{griebel10} developed 
a new general class of dimension-adaptive quadrature methods. The quadrature schemes detect and exploit the low effect dimensionality
of a function. This is achieved by truncation and discretization of the anchored-ANOVA
decomposition of the function being approximated. This method has been used successfully 
to estimate high-dimensional integrals arising in finance.

In addition to dimension based adaptivity, efficiency in approximation of high-dimensional functions can also be obtained through local adaptation. 
Locally adaptive sparse grids were first use by Griebel~\cite{griebel98} to solve the solution of partial differential equation and have also been used
used to quantify uncertainty in mathematical models~\cite{ma09}, interpolate and integrate functions~\cite{pfluger10}
and scattered data approximation problem~\cite{pfluger10-thesis}. 
The refinement of the sparse grid is guided by the magnitude of the so-called
hierarchical surplus, which is the difference between the true function and the approximation at a new grid point 
before that point is used in the interpolation. When a grid point is identified for refinement, new points are invested locally around that point in every dimension.
These new points are subsequently refined. The resulting method automatically concentrates function evaluations in rapidly varying or 
discontinuous regions. 

The aforementioned locally adaptive sparse grid methods are implicitly dimension-adaptive, but often points are constructed unnecessarily 
in `unimportant' dimensions. In comparison, the generalised sparse grid algorithm~\cite{gerstner03} provides 
an efficient means of identifying the effective dimensionality of a problem and restricting function 
evaluations to that sub-dimensional space. This method performs extremely well when the solution is smooth. 
However the efficiency of the generalised sparse grid method can be significantly improved when 
only small regions of the input space contribute to the model's variability. 

In this paper
we combine local refinement with the dimension-adaptive algorithm of the generalised sparse grid method
to construct an efficient high-dimensional interpolation method. Furthermore we utilise the localised polynomial 
basis proposed by Bungartz~\cite{bungartz98} to create a higher-order method 
which achieves fast rates of convergence in smooth regions and accuracy, comparable to linear methods, 
around discontinuities. 
We coin this approach the $h$-Adaptive Generalised Sparse Grid ($h$-GSG) method. 
The convergence of the proposed method is analysed, with respect to the order of the local polynomial basis and the dimensionality of the input 
space.

\section{Adaptive sparse grids}
In this paper we will attempt to interpolate and integrate functions $f:\Omega\rightarrow \mathbb{R}$ defined on a $d$-dimensional bounded domain $\Omega$.
We need not know the closed form of $f$, we only require that the function $f$ can be evaluated at arbitrary points in $\Omega$ using a numerical code. 

\subsection{Interpolation}
\label{sec:subspaceSpiltting}
To construct an interpolant of $f$, we must first discretize $\Omega$. 
Without loss of generality let us consider functions defined on the $d$-dimensional unit hypercube $\Omega=[0,1]^d$. 
Sparse grids are a direct sum of anisotropic grids $\Omega_\bi$ on the domain $\Omega$ 
where $\bi=(i_1,\ldots,i_d)\in \mathbb{N}^d$ is a multi-index denoting the level of refinement 
of the grid in each dimension $d$.  Each grid $\Omega_\bi$ is a tensor product of one-dimensional grids 
\begin{equation}\label{eq:sg_1Dgrid}
\Omega_i=(x_{i,1},\ldots,x_{i,m_i})
\end{equation}
where $m_i$ is odd and represents the number of points $x_{i,j}$ in the $i$th level one dimensional grid. Specifically
\[\Omega_\bi=\bigotimes_{n=1}^d \Omega_{i_n} \]
which consists of the points $\bx_{\bi,\bj}=(x_{i_1,j_1},\ldots,x_{i_d,j_d})$ and 
where $\bi$ indicates the level of refinement and $\bj$ denotes the location of a given grid point. 
The exact coordinates of each point and the total number of points $m_{i_1}\times\cdots\times m_{i_d}$ 
is dependent on the type of one-dimensional grids used. 

Each grid $\Omega_\bi$ is associated with a discrete approximation space $V_\bi$ and a set of basis functions 
that span the discrete space. The types of basis functions that can be used are dependent on the type 
of one-dimensional grids employed. The most frequently used and simplest choice 
are the multi-linear piecewise basis functions~\cite{bungartz04,griebel98,ma09},
based upon the one-dimensional formula
\begin{equation*}
\Psi_{i,j}(x)=
\begin{cases} 1-(m_i-1)|x-x_{i,j} | & \text{if $|x-x_{i.j}|<h_i$}\\
0 &\text{otherwise}
\end{cases}
\end{equation*}
centered at the points
\begin{equation*}
\label{eq:equidistantPoints} 
 x_{i,j}=
\begin{cases} 
j\times h & i\ge 1\textrm{ and }0\le j\le m_i\\
0.5 & i=1          
\end{cases}
\end{equation*}
where $h_i=1/(m_i-1)$.

These 1D basis functions can be used to form a set of $d$-dimensional basis functions 
\[
 \Psi_{\bi,\bj}(\bx)=\prod_{n=1}^d \Psi_{i_n,j_n}(x_n)
\]
which span the discrete space $V_\bi$. Specifically
\begin{equation}\label{eq:sg_fullspace}
 V_\bi=\text{span}\left\{ \Psi_{\bi,\bj}\,|\, j_n=1,\ldots,m_{i_n},\, n=1,\ldots,d\right\}
\end{equation}
The spaces $V_\bi$ can be used to define hierarchical difference spaces $W_\bi$
\begin{equation}\label{eq:sg_diffspaceconst}
 W_\bi=V_\bi\setminus\bigoplus_{n=1}^d V_{\bi-\mathbf{e}_n}
\end{equation}
These spaces consist of all the basis functions $\Psi_{\bi,\bj} \in V_\bi$ with associated points $\bx_{\bi,\bj}$ 
that are not associated with any of the basis functions in spaces smaller than $V_\bi$. A discrete space $V_\bk$
is smaller than a space $V_\bi$ if $\bk\le\bi$. 
Setting $V_\bi=0$ and using~\eqref{eq:sg_fullspace} and~\eqref{eq:sg_diffspaceconst}
we obtain
\begin{equation*}
 W_\bi=\text{span}\left\{ \Psi_{\bi,\bj}\,|\, \bj\in B_\bi \right\}
\end{equation*}
where 
\begin{equation}
\label{eq:BIndexSet}
B_\bi=\{\bj:j_n=1,\ldots,m_{i_n},\, j\text{ odd } n=1,\ldots,d\}
\end{equation}
These hierarchical difference spaces can be used to decompose the input space $\Omega=V$ such that 
\begin{equation*}
 V=\bigoplus_{k_1=0}^{\infty}\cdots\bigoplus_{k_d=0}^{\infty} W_\bk=\bigoplus_{\bk\in\mathbb{R}^d}W_\bk
\end{equation*}
For numerical purposes we must truncate the number of difference spaces used to construct $V$ to some level $l$. 
Specifically the classical finite dimensional 
sparse grid space is defined by
\begin{equation}
\label{eq:sg_vspace}
V_{l,d}^{(1)}=\bigoplus_{|\bi|_1\le l}W_\bi
\end{equation}

With such a decomposition any function $f(\bx)\in V$ can be approximated by
\begin{equation}\label{eq:sgInterpolant}
 f_{l,d}(\bx)=\sum_{|\bi|_1\le l} \sum_{\bj\in B_\bi} v_{\bi,\bj} \, \Psi_{\bi,\bj}(\bx)
\end{equation}
where $v_{\bi,\bj}\in \mathbb{R}$ are the coefficient values of the hierarchical product 
basis, also known as the hierarchical surplus.

The size of the sparse grid space is 
$\lvert V_{l,d}^{(1)} \rvert=\cO(h_l^{-l}\cdot\lvert\log_2 h_l\rvert^{d-1})=\cO(2^l\cdot l^{d-1})$
which is a significant reduction on the $\cO(2^{l\cdot d})$ number of points required by the full tensor 
product space $V_{l,d}^{(\infty)}$ obtained by choosing $|\bi|_\infty=\max_{0\le n\le d} i_n\le l$.

\subsection{Quadrature}
The extension from interpolation to quadrature is straightforward. 
We can approximate  the integral of a function $f$
\begin{equation*} 
I[f(\bx)]=\int_{I_{\bx}}f(\bx)\,d\mu(\bx)
\end{equation*}
using the hierarchical sparse grid interpolant~\eqref{eq:sgInterpolant}. Utilizing this formulation 
these integrals can be approximated by
\begin{eqnarray*}
I[f_{l,d}(\bx)]&=&\int_{I_{\bx}} \sum_{|\bi|_1\le l}\sum_{\bj\in B\bi}v_{\bi,\bj}\,\Psi_{\bi,\bj}(\bx)\,d\mu(\bx)\\
&=&\sum_{{|\bi|_1\le l}}\sum_{\bj\in B\bi}v_{\bi,\bj}\, w_{\bi,\bj}
\end{eqnarray*}
where the weights 
\begin{equation*}
w_{\bi,\bj}=\int_{I_{\bx}}\Psi_{\bi,\bj}(\bx)\,d\mu(\bx)
\end{equation*}
can be calculated easily and with no need for extra function evaluations once the interpolant has been constructed.
One simply needs to store the volumes of the high-order
basis functions. For $d\mu=d\bx$ these volumes can be calculated analytically.

\subsection{A Local High-Order Basis}
\label{sec:highOrderBasis}
Sparse grids are not restricted to piecewise multi-linear basis functions that are constructed on equidistant grids. 
Various formulations exist. In the following we propose a high order local basis for interpolation and quadrature. 
This basis was first proposed by Bungartz~\cite{bungartz98} for the solution of partial differential equations. 
The local nature 
of the basis functions allows for local adaptivity and restricts the effects of Gibbs type phenomenon 
experienced by global polynomial approximation whilst still achieving polynomial convergence in smooth regions.

As with the linear case, we restrict our attention to grids $\Omega_i$ with mesh spacing  
that is equidistant with respect to each individual dimension but may vary between dimensions.
Let $\Psi^{(p)}_{i,j}$ denote a one-dimensional polynomial of degree $p$ defined on the interval $[x_{i,j}-h_i,x_{i,j}+h_i]$. This localized support is 
essential for application of the high-order basis to non-smooth problems and the ultimate goal of an adaptive method. 
Uniquely defining this basis requires $p+1$ conditions that $\Psi^{(p)}_{i,j}$ must satisfy. 

Here we take advantage of 
the fact that each point $\bx_{\bi,\bj}$ has an ancestry. 
The one-dimensional equidistant points $x_{i,j}$ can be considered as a tree-like data structure. 
The coordinate of each point is defined uniquely by the level~$i$ and the 
position~$j$. With this observation we can define a local $p$-th order polynomial using the points 
$x_{i,j}-h_i$, $x_{i,j}$, $x_{i,j}+h_i$ and the next $p-2$ closest hierarchical ancestors of $x_{i,j}$. 
\begin{definition}
Given the one-dimensional grid $\Omega_i$ with grid points defined according to \eqref{eq:equidistantPoints} 
the $p$-th order basis function $\Psi_{i,j}^{(p)}$ is the hierarchical interpolant of the point 
$x_{i,j}-h_i$, $x_{i,j}$, $x_{i,j}+h_i$ and the next $p-2$ closest hierarchical ancestors of $x_{i,j}$ restricted to 
the local support $[x_{i,j}-h_i, x_{i,j}+h_i]$. Specifically by renaming 
all the points except $x_{i,j}$ in ascending order as $x_0,\ldots,x_p$ 
the piecewise Lagrange basis can be written
\[
 \Psi^{(p)}_{i,j}(x)=\begin{cases} \prod_{k=0}^{p}\frac{x-x_k}{x_{i,j}-x_k}& \text{if $|x-x_{i,j}|<h_i$}\\
 0 &\text{otherwise}
\end{cases}
\] 
\end{definition}

 The order $p$ of the basis function is dependent
on the hierarchical level $i$ of $x_{i,j}$. For $p>2$, $p+1$ ancestors are needed to construct the basis $\Psi_{i,j}$. 
On level one,
only one ancestor ($x=0.5$, on level zero) is available and thus only linear basis functions can be used. 
On level two, only two ancestors exist, and thus only linear or quadratic basis functions can be used. 
Subsequently basis functions of degree $p$ can only 
be used when $ i \ge p$. This represents a slight modification of the approach employed by Bungartz~\cite{bungartz98}
who allowed basis functions of degree $p$ to be used when $ i \ge p-1$. Bungartz approach was designed for sparse grids
with homogeneous boundary conditions. 

Throughout the remainder of this manuscript we restrict our attention to basis functions with fixed maximum degree $p_\textrm{max}$. That is,
the order of the basis is increased with each level of the sparse grid until the order of the basis is $p_\textrm{max}-1$. 
On all subsequent levels the order of the basis is
kept constant at $p_\textrm{max}$. The tensor product construction of the multi-dimensional basis means that the degree $\mathbf{p}=(p_i,\ldots,p_d)$ of a $d$-dimensional basis function $\Psi_{\bi,\bj}$ must satisfy 
\[
 0 \le p_n =\min \{p_\textrm{max},i_n\},\quad i_n\ge0 \quad, n=1,\ldots,d
\]
Here $p_n=0$ represents the constant function
centred at the midpoint of $[0,1]$.

\subsection{Adaptivity}
\label{sec:h-adaptivity}
The classical sparse grids presented in Section~\ref{sec:subspaceSpiltting} are based upon the index set
\[
 \cI=\{\bi\in\mathbb{N}^d:\lvert\bi\rvert_1\le l\}
\]
This construction delays the curse of dimensionality by assuming that the importance of any interaction 
between a subset of a function's variables decreases as the number of variables involved in the interaction 
(interaction order) increases. Although an advance on full tensor product spaces, such approximations can still be
improved. The classical sparse grid construction treats all dimensions equally and all interactions of the same order
equally. In practice, often only a small subset of variables and interactions
contributes significantly to the variability of the function $f$. Moreover, frequently only small regions within
the input space possess high variability. In some cases the important dimensions, 
interactions and regions can be
determined a priori, but in most cases these properties must be identified during the computational procedure.

The generalised sparse grid method~\cite{gerstner03} is extremely effective at determining the dimensions
and interactions that contribute significantly
to the function variability, according to some predefined measure. However the efficiency of this method deteriorates
when a large proportion of the function variability is concentrated in small regions of the input space.
In contrast to the generalised sparse grid method, locally-adaptive methods, such as that of Ma and Zabaras~\cite{ma09}, attempt
to reduce the number of points in a sparse grid by concentrating refinement only in rapidly varying or discontinuous
regions. This method is implicitly dimension-adaptive but often points are constructed necessarily 
in `unimportant' dimensions~\cite{ma10}.

In this section we propose a method which combines the strengths of both 
local-adaptivity and the generalised sparse grid algorithm. 
We coin this approach the $h$-Adaptive Generalised Sparse Grid ($h$-GSG) method.

\subsubsection{Generalised Sparse Grid (GSG) Algorithm}
\label{sec:GSGAlgorithm}
Gerstner~\cite{gerstner03} generalised the sparse
grid construction by considering the index sets based upon the admissibility criterion
\begin{equation}
\label{eq:gsg_admissibility}
\bi-\mathbf{e}_j\in\mathcal{I}\text{ for }1\le j\le d,\, i_j > 1
\end{equation}
This so-called generalised sparse grid method~\cite{gerstner03} is extremely effective at 
determining the hierarchical difference spaces that contribute significantly
to the function variability, according to some predefined measure. 

The generalised sparse grid method is a greedy algorithm which attempts to find the index set $\mathcal{I}$ such that for a 
given number of points the approximation error is minimized. 
Starting with $\mathcal{I}=\{\mathbf{0}\}$ the index set is built iteratively by searching the forward neighbourhood of 
the current index set for new admissible indices. The forward neighbourhood of an index $\bi$ is the set of $d$ 
indices $\{\bi+\be_j:1\le j\le d\}$. Similarly, the backwards neighbourhood is just $\{\bi-\be_j:1\le j\le d\}$. 

Once the forward neighbourhood has been identified, each forward neighbour is checked for admissibility 
using~\eqref{eq:gsg_admissibility}. The grid points associated with each admissible index
are then evaluated and the error of these spaces calculated. The calculation of these errors will be addressed shortly.
The forward neighbour with the largest error is then added to the current index set $\cI$ and the set of admissible 
indices is updated.
  
To facilitate easy computation of new admissible indices we partition the index set $\cI$ 
into two disjoint sets $\cO$ and $\cA$, which Gerstner~\cite{gerstner03} refers to as the old and active index sets, 
respectively. 
The active index set $\cA$ contains all the indices in $\cI$ that have been constructed but whose forward neighbours
have not been considered. The old index set contains all the indices remaining in the current index set $\cI$. 
The algorithm proceeds by searching the forward neighbourhood of the index $\bi\in\cA$ with the largest error
for admissible indices. All the new admissible index sets are added to the active index set $\cA$ 
and the index $\bi$ is then added to the old index set $\cO$. This process is repeated until a global error 
is below a predefined tolerance $\varepsilon$. 

The exact error associated with each index $\bi$ is unknown. Consequently each time an index $\bi$ is 
deemed admissible an approximation of the error $r_\bi$ must be used. Numerous error criteria can be utilised. 
Here we employ the error measure
\[
 r_\bi=\left\lvert\sum_{\bj\in B_\bi} v_{\bi,\bj}\cdot w_{\bi,\bj}\right\rvert
\]
These index-based error criteria can be used to approximate the global error. We propose the following global error
indicator $r$
\[
 r=\sum_{\bi\in\cA} r_\bi
\]
When $r<\varepsilon$ the generalised sparse grid algorithm is terminated.
The generalised sparse grid algorithm is outlined in 
Algorithm~\ref{alg:generalisedSG}.
\begin{algorithm}
\caption{Generalised Sparse Grid Approximation}
\label{alg:generalisedSG}
\begin{algorithmic}
\STATE $\bi=(0,\ldots,0)$
\STATE $\cA:=\{\bi\}$
\STATE $\cO:=\emptyset$
\STATE $r:=r_\bi$
\WHILE {$r>\varepsilon$}
	\STATE select $\bi \in \cA$ with largest error indicator $r_{\bi}$
	\STATE $\cA:=\cA \setminus \{\bi\}$
	\STATE $\cO:=\cI\cup\{\bi\}$
	\STATE $r:=r-r_\bi$
	\FOR{$k:=1,\ldots,d$}
		\STATE $\bj=\bi+\mathbf{e}_k$
		\IF {$\bj-\mathbf{e}_n \in O \quad \forall\; n=1,\ldots,d$}
			\STATE $\cA:=\cA\cup\{\bj\}$
			\STATE CreateGrid($\bj$)
			\STATE $r:=r+r_\bj$
		\ENDIF
	\ENDFOR
\ENDWHILE
\end{algorithmic}
\end{algorithm} 

Three steps of the generalised sparse grid algorithm depicting the construction of the sparse grid index set
are shown in Figure~\ref{fig:gridIndexing}. The top row represents the current index sets. The
bottom row depicts the corresponding sparse grid. At each step the forward neighbours of the grid index $\bi$ 
 with the largest error $r_\bi$ (striped box) are 
checked for admissibility. A forward neighbour is admissible if all indices in its backwards neighbourhood are in 
the old index set (grey boxes). All admissible indices (pointed to by an arrow) are added to the active index set (black and striped boxes).

The striped box $\bi=(1,1$) in the first step has two admissible neighbours as the backwards neighbourhoods of both
forward neighbours are complete. In comparison the striped box $\bi=(2,1)$ in the second step only has one admissible 
index. The index $\bj_1=(3,1)$ has two backwards neighbours $\bj_1-\be_1=(0,2)$ and $\bj_1-\be_2=(1,1)$ in the old index set,
and thus is admissible. In contrast the index $\bj_2=(2,2)$ has one backwards neighbour in the old index set $\bj_2-\be_1=\bj_2$
and one in the active set $\bj_2-\be_2=(2,1)$, and so is not admissible.

\begin{figure}
\begin{center}
\includegraphics[width=0.85\textwidth]{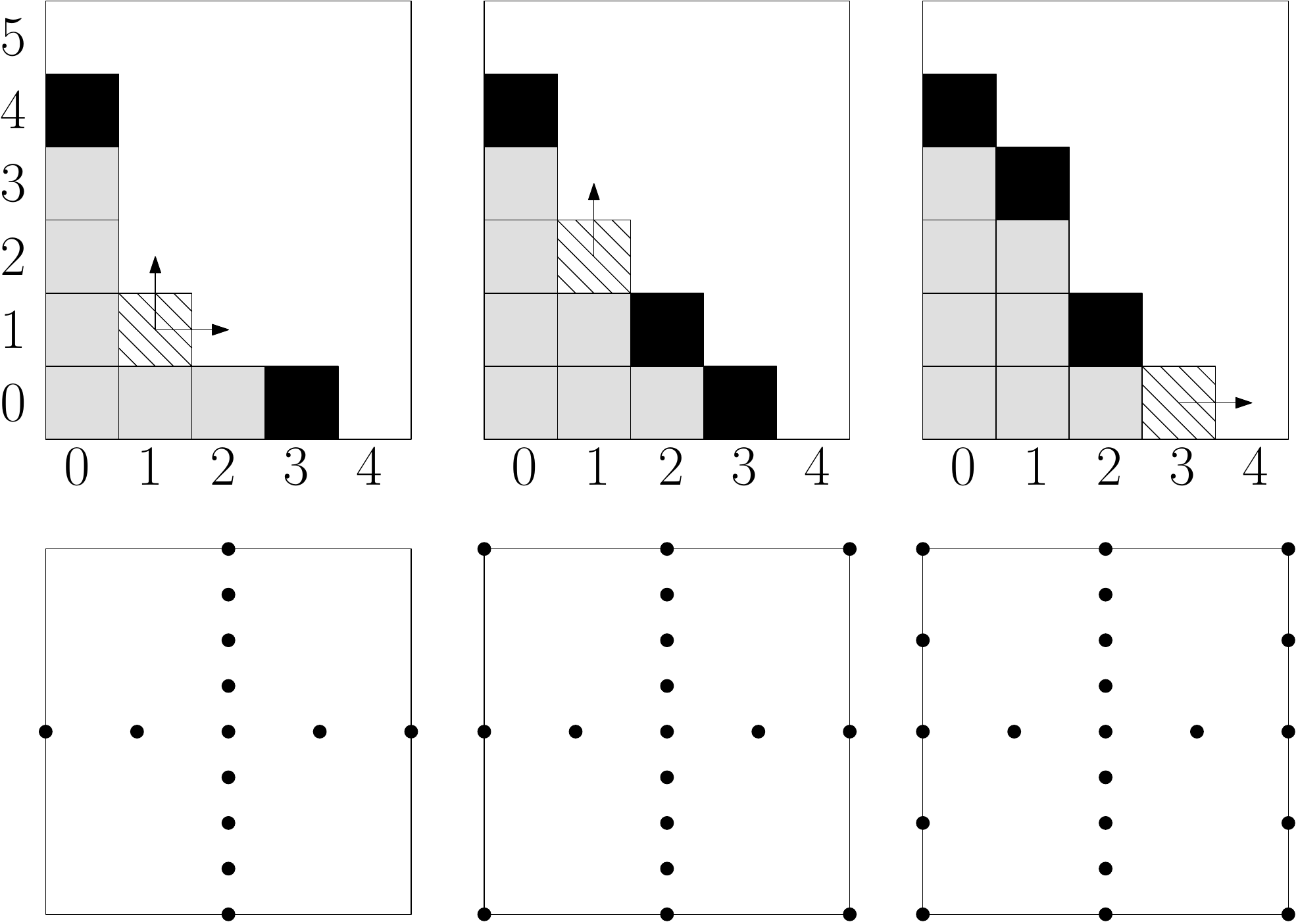}
\caption{Three steps of the generalised sparse grid algorithm. The top row represents the current index sets. 
 Active grid indices are in black, 
indices in the old index set $\cO$ are in grey and the active index with the largest error indicator is
striped. The bottom row depicts the corresponding sparse grid. Only the  points associated with indices in the old index
set are shown.}

\label{fig:gridIndexing}
\end{center}
 \end{figure}

Although we wish to use the generalised sparse grid algorithm for interpolation the error criteria we have proposed
are based upon an integral formulation. Specifically the error indicator $r_\bi$ measures the contribution of
the index $\bi$ to the global integral approximation. Furthermore the algorithm is terminated when the approximated 
error $r$ in the integral is below a predefined threshold $\varepsilon$. This choice was made purposefully. 

The magnitude of the hierarchical surplus, which is the size of the difference between the true function and the sparse
grid approximation at a grid point, may be more synonymous with interpolation. Simply adding indices with 
large hierarchical surpluses, however, is inefficient.  The magnitude of the hierarchical surpluses decays slowly in regions adjacent 
to discontinuities. At the site of jump discontinuities the hierarchical surplus will be at best half of the magnitude of the jump, for any finite number of grid points. Thus the algorithm can proceed much further than is
necessary. The use of the error criterion $r_\bi$ provides a lower bound on the size of the support of 
the basis functions used by weighting the magnitude of the hierarchical surplus by the probability that an arbitrary point $\bx$
will fall within its support.

\subsubsection{Regional Adaptivity}

Each time a grid index $\bi$ is added to the active index set, the traditional generalised sparse grid algorithm
evaluates all the points in the set $B_\bi$. 
Such an approach is inefficient if a large proportion of the function variability is 
concentrated in small regions of the input space. When using equidistant grids, the creation of the grid $\bi$ 
requires approximately two times the number of grid points (and thus function evaluations) than those necessary
to construct the index $\bi-\be_j$. 
Consequently we propose introducing a locally adaptive procedure to construct the points associated with each grid index.
\begin{figure}
\begin{center}
 \includegraphics[width=0.5\textwidth]{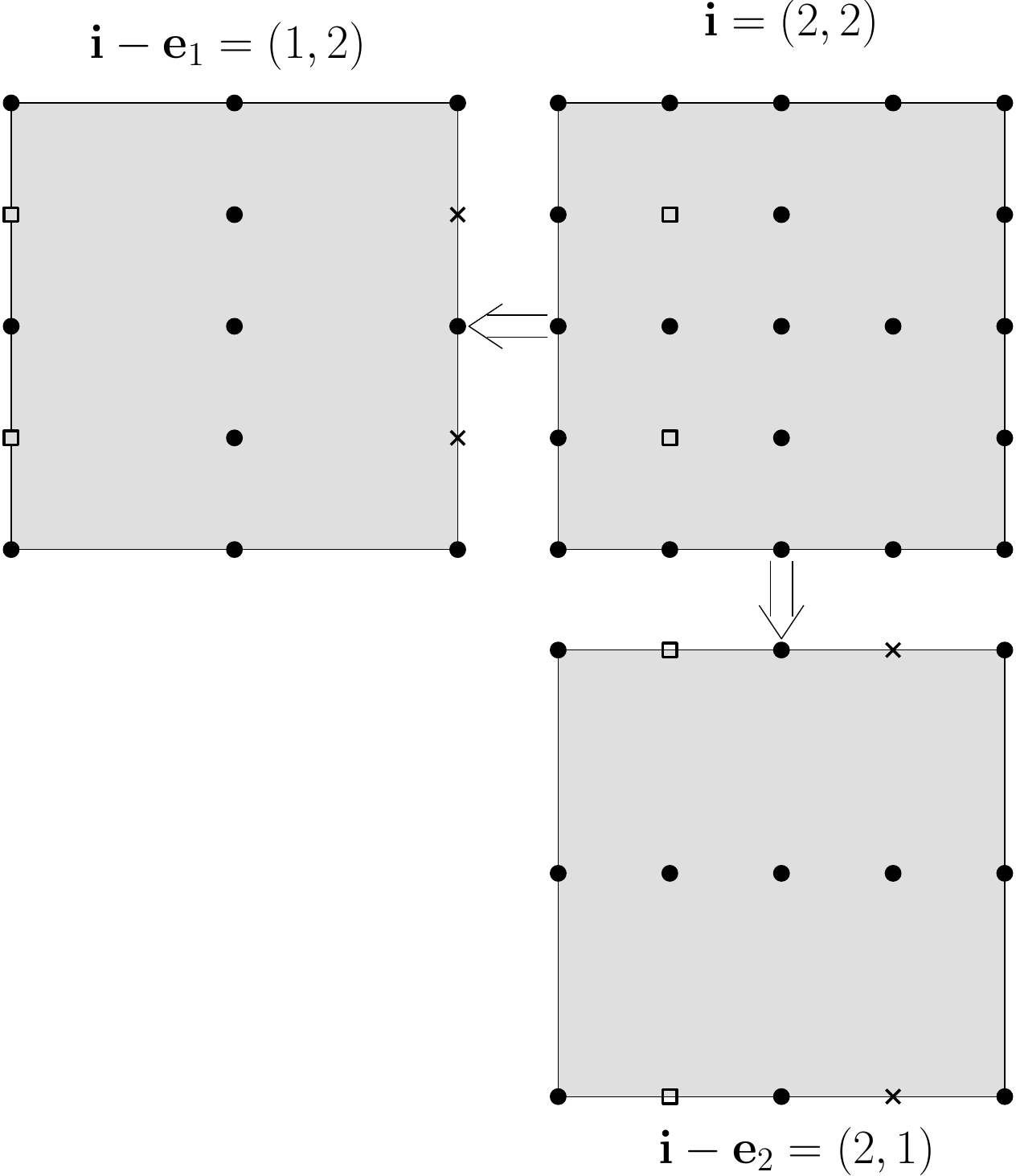}
\caption{An example of local adaptation integrated with the generalised sparse grid algorithm. Assume
that the function only varies significantly in the left half of the domain. 
The top left and bottom right grids are the backwards neighbours of the grid being created. 
Circles represent points in the sparse grid, 
squares are points in the active point sets, and crosses are points in the redundant point sets. 
The active (square) points in the backwards neighbours are refined to produce the set of new points
that must be added. In this example only two new points (squares in grid $\bi=(2,2)$) are added. }

\label{fig:hGSGPointUpdate}
\end{center}
 \end{figure}
\label{sec:hGSGAdaptivity}
To incoporate local adaptivity into the generalised sparse grid algorithm we define the two sets $\cA_\bi$ and 
$\cR_\bi$ for each grid index $\bi$. We refer to these sets respectively as the active point set and redundant point set of the
grid index $\bi$. The active point set $\cA_\bi$ contains all admissible points associated with the index $\bi$ 
with an error indicator $\gamma_{\bi,\bj}\ge\varepsilon$. The redundant point set $\cR_\bi$
contains all admissible points with $\gamma_{\bi,\bj}<\varepsilon$. A point is admissible if one of its $d$ possible 
ancestors exists in the grids associated with the backwards neighbourhood of $\bi$.
If, and only if, a grid point is admissible it is created (the function is evaluated) and the error indicator 
$\gamma_{\bi,\bj}$ calculated. This drastically reduces the number of points generated when a new grid index is created.

To guide local refinement we propose using the error indicator
\begin{equation}
\label{eq:absolute-error-criterion}
\gamma_{\bi,\bj}=\lvert v_{\bi,\bj}\cdot w_{\bi,\bj}\rvert 
\end{equation}
If $\gamma_{\bi,\bj}\ge\varepsilon$ the point $\bx_{\bi,\bj}$ is added to the active point set $\cA_\bi$, 
otherwise it is added to the redundant index set $\cR_\bi$. 
The procedure used to implement $h$-adaptivity on each grid index is outlined in Algorithm~\ref{alg:createGrid}.
\begin{algorithm}
\caption{CreateGrid($\bi$)}
\begin{algorithmic}
\label{alg:createGrid}
\FOR{($n\in\{1,\ldots,d\}$)}
  \FOR {($\bx_{\bi-\be_n,\bj} \in \cA_{\bi-\be_n}$ )}
	\STATE $\mathcal{C}$=FindAxialChildren($\bx_{\bi-\be_n,\bj}$,$n$)
	\FOR {( $\bx_{\bi,\bk}\in\mathcal{C}$ )}		
		\IF {( $ \gamma_{\bi,\bk} \ge \varepsilon$ )}
			\STATE $\cA_{\bi}:=\cA_\bi\cup\{\bx_{\bi,\bk}\}$
		\ELSE 
			\STATE $\cR_\bi:=\cR_\bi\cup\{\bx_{\bi,\bk}\}$
		\ENDIF
	\ENDFOR
  \ENDFOR
\ENDFOR
\end{algorithmic}
\end{algorithm}
 
Figure~\ref{fig:hGSGPointUpdate} shows an example of $h$-adaptivity integrated with the generalised sparse grid 
algorithm. Here the grid index $\bi=(2,2)$ has been deemed admissible by the generalised sparse grid algorithm.
Both the backwards neighbours ($\bi-\be_1=(1,2)$ and $\bi-\be_1=(2,1)$) exist in the old index set $\cO$. The active points in the backwards
neighbours are used to determine which points in the active index $\bi$ should be evaluated. For any point in the 
active set of the $n$-th backwards neighbour $n=1,\ldots,d$ the children of that point are created in the $n$-th
axial direction. This refinement is carried out for all points in the active point set of the backward neighbour and
for all backwards neighbours.   

\subsubsection{Efficient Termination}
\label{sec:efficientTermination}
The generalised sparse grid (GSG) algorithm is a greedy algorithm which efficiently identifies the sparse grid index
set that is necessary to interpolate a function up to a level of predefined accuracy. The algorithm can determine
the number of variable interactions and the individual importance of each variable~\cite{griebel10}. This is achieved
by successively adding the grid index with the largest error indicator to the old index set and 
searching its forward neighbourhood for admissible indices. Every admissible index is added (and thus created) to the active 
index set without regard for the error associated with that grid. 
The algorithm finally terminates when $\sum_{\bi\in\cA} r_\bi<\varepsilon$. 

The decision to add all indices $\bi$, regardless
of the size of their associated error indicator $r_\bi$, typically results in the creation of a large number of
grids with $r_\bi<<\varepsilon$ and which have little effect on the accuracy of the approximation. To reduce
the number of these unimportant indices we propose only adding admissible indices with $r_\bi\ge\varepsilon$ to
the active index set $\cA$. This significantly reduces the number of grid indices in the final index set $\cI$
and thus the total number of function evaluations, with only minor effect on the overall accuracy of the 
generalised sparse grid method.
\section{Error Analysis}
\label{sec:sgErrorAnalysis}
In this section we derive a bound on the error of the proposed $h$-GSG method. 
For ease of discussion let us rewrite~\eqref{eq:sgInterpolant} in the following form
\begin{equation}
\label{eq:sg_representation2}
f_{l,d}(\bxi)=\sum_{\bi\le l} f_\bi(\xi),\quad f_\bi(\xi)=\sum_{\bj\in B_\bi} v_{\bi,\bj}^{(p)}\cdot \Psi_{\bi,\bj}^{(p)} \in W_\bi
\end{equation}

The proposed $h$-GSG algorithm terminates when all points in the sparse grid with an error indicator  $\gamma_{\bi,\bj}\ge\varepsilon$
have been considered. This truncation of the sparse grid space has an effect on the accuracy of the approximation.
This effect is quantified by the following theorem. 
\begin{theorem}
\label{thm:optErr}
Let $f_{\varepsilon,\mathrm{opt}}$ be an interpolation of $u$ that obtains $\lVert f-f_{\varepsilon,\mathrm{opt}}\rVert_q\le\varepsilon$ 
with the least number of function evaluations. Then for any function $u$ and a given tolerance $\varepsilon>0$  and the error criterion $\gamma_{\bi,\bj}=\left\lVert v_{\bi,\bj}\cdot \Psi_{\bi,\bj}\right\rVert_q $, 
the $h$-GSG approximation $f_{\varepsilon,d}$ satisfies
\[
 \lVert f-f_{\varepsilon,d} \rVert_q\le\varepsilon\left(1+N(\varepsilon)\right)
\]
where $N(\varepsilon)$ 
is the number of points in the optimal interpolant but not in the  $h$-GSG interpolant.
\end{theorem}
\begin{proof}
Let 
\[
f_{\varepsilon,\mathrm{opt}}=\sum_{(\bi,\bj)\in P_{\varepsilon,\mathrm{opt}}} v_{\bi,\bj}\cdot\Psi_{\bi,\bj}
\]
be an interpolation of $f$ that obtains $\lVert f-f_{\varepsilon,\mathrm{opt}}\rVert_q\le\varepsilon$.
The points in this optimal approximant are defined by the index set
\[
 P_{\varepsilon,\mathrm{opt}}:=\{(\bi,\bj)\;:\;\bi\in\cI_{\varepsilon,\mathrm{opt}}\;\mathrm{and}\;\bj\in B_{\bi,\varepsilon,\mathrm{opt}}\subseteq B_{\bi}\}
\]
Similarly denote the $h$-GSG interpolant with an error indicator $\gamma_{\bi,\bj}$ by
\[
f_{\varepsilon,d}=\sum_{(\bi,\bj)\in P_{\varepsilon,d}} v_{\bi,\bj}\cdot\Psi_{\bi,\bj}
\]
where the point indices in the $h$-GSG approximant are 
\[
 P_{\varepsilon,d}:=\{(\bi,\bj)\;:\;{\bi\in\cI_{\varepsilon,d}},\;\bj\in B_\bi\;\mathrm{and}\;|v_{\bi,\bj}|\ge\varepsilon\}
\]
Now denote $P^\mathrm{common}:= (P_{\varepsilon,d}\bigcap P_{\varepsilon,\mathrm{opt}})$ the set of indices common to both the 
optimal and $h$-GSG approximants and denote $P^\mathrm{unique}:= ((P_{\varepsilon,d}\bigcup P_{\varepsilon,\mathrm{opt}})
\setminus P^\mathrm{common})$ the set of indices that exist only in $ P_{\varepsilon,\mathrm{opt}}$ or $P_{\varepsilon,d}$.

We can split $P^\mathrm{unique}$ further into $P_{\varepsilon,\mathrm{opt}}^{\mathrm{unique}}$ and 
$P_{\varepsilon,d}^{\mathrm{unique}}$ which are points unique to the optimal approximant and the $h$-GSG approximants 
respectively. Using this splitting and the linearity of the hierarchical interpolants $f_{\varepsilon,\mathrm{opt}}$ 
and $f_{\varepsilon,d}$, yields
\[
 f_{\varepsilon,\mathrm{opt}}=f_{\varepsilon,\mathrm{opt}}^{\mathrm{unique}} + f^{\mathrm{common}}\quad\mathrm{and}\quad
f_{\varepsilon,d}=f_{\varepsilon,d}^{\mathrm{unique}}+f^{\mathrm{common}}
\]
where
\[
 f_{\varepsilon,\mathrm{opt}}^{\mathrm{unique}}=\sum_{(\bi,\bj)\in P_{\varepsilon,\mathrm{opt}}^{\mathrm{unique}}} v_{\bi,\bj}\cdot\Psi_{\bi,\bj}
 \quad,\quad f_{\varepsilon,d}^{\mathrm{unique}}=\sum_{(\bi,\bj)\in P_{\varepsilon,d}^{\mathrm{unique}}} v_{\bi,\bj}\cdot\Psi_{\bi,\bj} 
\]
and
\[
f^{\mathrm{common}}=\sum_{(\bi,\bj)\in P^{\mathrm{common}}} v_{\bi,\bj}\cdot\Psi_{\bi,\bj} 
\]
Using these definitions we can write
\begin{eqnarray}
\label{eq:optDecomp}
 \lVert f-f_{\varepsilon,d} \rVert_q&=& \lVert f-f_{\varepsilon,\mathrm{opt}} +  f_{\varepsilon,\mathrm{opt}}-f_{\varepsilon,d} \rVert_q\nonumber\\
 &=&\lVert f-f_{\varepsilon,\mathrm{opt}} +  
 (f_{\varepsilon,\mathrm{opt}}^{\mathrm{unique}} + f^{\mathrm{common}})-(f_{\varepsilon,d}^{\mathrm{unique}}+f^{\mathrm{common}}) \rVert_q\nonumber\\
 &\le& \lVert f-(f_{\varepsilon,\mathrm{opt}}+f_{\varepsilon,d}^{\mathrm{unique}})\rVert_q+\lVert f_{\varepsilon,\mathrm{opt}}^{\mathrm{unique}}\rVert_q
\end{eqnarray}
Assuming that the adaptivity of the $h$-GSG method works perfectly, that is 
\[
 \gamma_{\bi,\bj}=\left\lVert v_{\bi,\bj}\cdot \Psi_{\bi,\bj}\right\rVert_q \le \varepsilon,\quad\forall\; (\bi,\bj)\notin P_{\varepsilon,d}^{\mathrm{unique}} 
\]
then
\begin{eqnarray}
  \lVert f_{\varepsilon,\mathrm{opt}}^{\mathrm{unique}}\rVert_q&=&
  \left\lVert\sum_{(\bi,\bj)\in P_{\varepsilon,\mathrm{opt}}^{\mathrm{unique}}} v_{\bi,\bj}\cdot\Psi_{\bi,\bj}\right\rVert_q\nonumber\\
  &\le&\sum_{(\bi,\bj)\in P_{\varepsilon,\mathrm{opt}}^{\mathrm{unique}}} \left\lVert v_{\bi,\bj}\cdot\Psi_{\bi,\bj}\right\rVert_q\nonumber\\
\label{eq:optErr}
&\le&\#(P_{\varepsilon,\mathrm{opt}}^{\mathrm{unique}})\cdot \varepsilon
\end{eqnarray}
By definition of the optimal interpolant
\begin{equation}
\label{eq:hgsgErr} 
\lVert f-(f_{\varepsilon,\mathrm{opt}}+f_{\varepsilon,d}^{\mathrm{unique}})\rVert_q\le\varepsilon
\end{equation}
Setting $N(\varepsilon)=\#(P_{\varepsilon,\mathrm{opt}}^{\mathrm{unique}})$
we arrive at the assertion.\qed
\end{proof}
Theorem~\ref{thm:optErr} states that the accuracy of the $h$-GSG interpolant is dependent on the number of points with cumulative
$\gamma_{\bi,\bj}<\varepsilon$ that are not in the approximation but have $\gamma_{\bi,\bj}\approx\varepsilon$. 
The exact number $N(\varepsilon)$ of these points is dependent on the smoothness of the function being approximated. 
The smoother the function, that is the faster the
hierarchical coefficients decay, the smaller $N(\varepsilon)$ will be.
\section{Numerical Study}
\label{sec:SGNumericalStudy}
In this section we investigate the performance of the proposed $h$-GSG method when applied to a number of numerical examples. 
We analyze convergence, with respect to the order of the local polynomial basis and the dimensionality for functions of varying smoothness. 
First we consider a set of two-dimensional functions which
visually illustrates the effect of the choice of basis degree and the performance of local adaptivity. 
We then discuss the performance of different basis functions 
when applied to functions of differing regularity. 
The effect of the termination condition presented in Section~\ref{sec:efficientTermination} is also presented. 
Finally the utility of the proposed method is shown for high-dimensional 
approximation with hundreds of variables.

In the following we will consider the following four functions:
\begin{equation}
\label{eq:lineSing}
f_1(\bxi)=\frac{1}{\lvert 0.3-\xi_1^2-\xi_2^2\lvert+0.1},\quad \bxi\in[0,1]^2
\end{equation}
\begin{equation}
\label{eq:f3}
f_2^d(\boldsymbol{\xi})=\exp\left(-\sum_{i=1}^dc_i^{2}(\xi_i-w_i)^2\right),\quad \bxi\in[0,1]^d
\end{equation}
\begin{equation}
\label{eq:f4}
f_3^d(\boldsymbol{\xi})=\exp\left(-\sum_{i=1}^dc_i|\xi_i-w_i|\right),\quad \bxi\in[0,1]^d
\end{equation}
\begin{equation}
\label{eq:f5}
f_4^d(\boldsymbol{\xi})=
\begin{cases}
0 & \text{if $\xi_1>w_1$ or $\xi_2>w_2$}\\
\exp\left(\sum_{i=1}^dc_i\xi_i\right) & \text{otherwise}
\end{cases},\quad \bxi\in[0,1]^d
\end{equation}
Unless otherwise stated, the coefficients $w_i=0.5$, $i=1,\ldots,d$.  The choice of $c_i$
determines the effective dimensionality of the function and is defined differently for each problem.
Although smooth, the mixed derivatives of the Gaussian function $f_2^d$  
can become large and thus degrade performance if not compensated for by appropriate adaptivity. The discontinuities
in functions $f_3^d$ and $f_4^d$ also degrade, with increasing magnitude, the efficiency of isotropic methods and 
subsequently can highlight the strengths and weaknesses of any interpolation method.

In the following we will analyze convergence with respect to the following measures:
\[
 \varepsilon_{\ell^\infty}=\max_{i=1,\ldots,N}|f(\xi_i)-g(\xi_i)|
\]
\[
 \varepsilon_{\ell^2}=\left(\frac{1}{N}\sum_{i=1}^{N}|f(\xi_i)-g(\xi_i)|^2\right)^{1/2}
\]
where $f$ and $g$ are the true function and approximation respectively. In all the following examples $N=1000$. 
Error in the quadrature rule $I_\mathrm{approx}$ is also considered 
and measured by
\[
 \varepsilon_\mathrm{integral}=\frac{I_\mathrm{approx}-I_\mathrm{exact}}{I_\mathrm{exact}}
\]
where $I_\mathrm{exact}$ is the exact integral. Unless otherwise stated this value is calculated analytically.

\subsection{A two-dimensional example}
Let us first consider two low dimensional functions, \eqref{eq:lineSing} and \eqref{eq:f4}, defined on the unit hypercube $[0,1]^2$. 
Figure~\ref{fig:genz-f4-d2} depicts the grids generated when the proposed method is applied to the piecewise continuous function~\eqref{eq:f4} using linear basis functions~(a) 
and quadratic basis functions~(b).
This function has discontinuities in its
first derivatives along $\xi_1=0.5$ and $\xi_2=0.5$. The linear $h$-GSG grid concentrates grid points around the rapidly varying region associated with the
discontinuous change in the derivative information. In comparison the quadratic basis requires significantly less 
function evaluations. The quadratic basis is able to obtain second order convergence in each of the four 
smooth quadrants whilst still approximating well at the discontinuities. 
Here $c_i=10/2^{i+2}$, and the absolute error criterion~\eqref{eq:absolute-error-criterion} is used with $\varepsilon=10^{-6}$.

The function~\eqref{eq:f4} possesses discontinuities that lie along the axial directions.
Let us now consider function~\eqref{eq:lineSing} which possesses a singularity that passes through both axial directions. 
Figure~\ref{fig:lineSing} depicts
the grids obtained using a tolerance of $\varepsilon=10^{-6}$ and
linear~(a) and quadratic basis functions~(b). The linear $h$-GSG grid concentrates grid points around the rapidly varying region associated with the
discontinuous change in the derivative information. Unlike the previous example it is now unclear whether the use of the quadratic basis function results in increased
efficiency. The accuracy of the linear basis functions is higher ($\varepsilon_{\ell^2}=3.19\cdot 10^{-3}$) 
than when the quadratic basis is used ($\varepsilon_{\ell^2}=1.15\cdot10^{-2}$) but the linear method requires almost three times as many points. The effect of varying the degree of 
the local basis is discussed in the following section. 

\begin{figure}[htp]
\centering
\begin{subfigmatrix}{2}
\subfigure[$p=1$, ($2477$ points, $\varepsilon_{\ell^2}=1.18\cdot10^{-4}$)]{\includegraphics[width=0.49\textwidth]{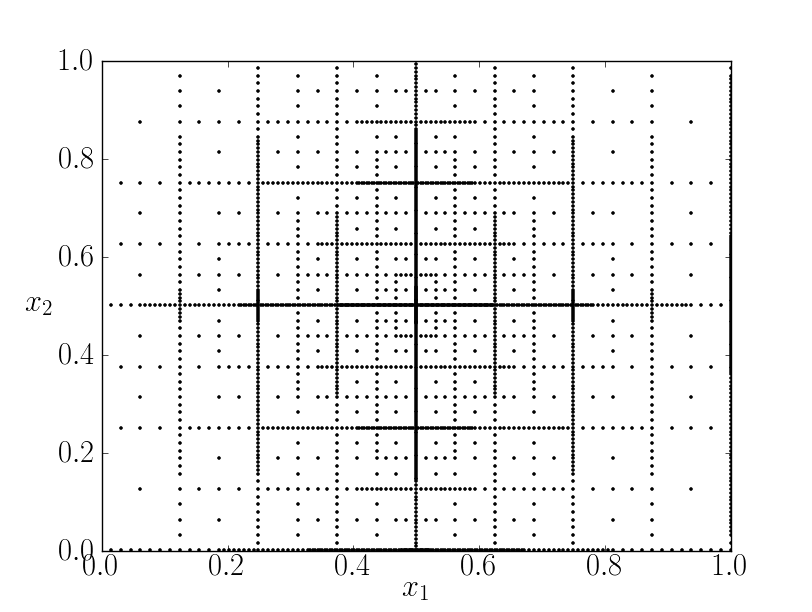}}
\subfigure[$p=2$, ($1257$ points, $\varepsilon_{\ell^2}=4.67\cdot10^{-5}$)]{\includegraphics[width=0.49\textwidth]{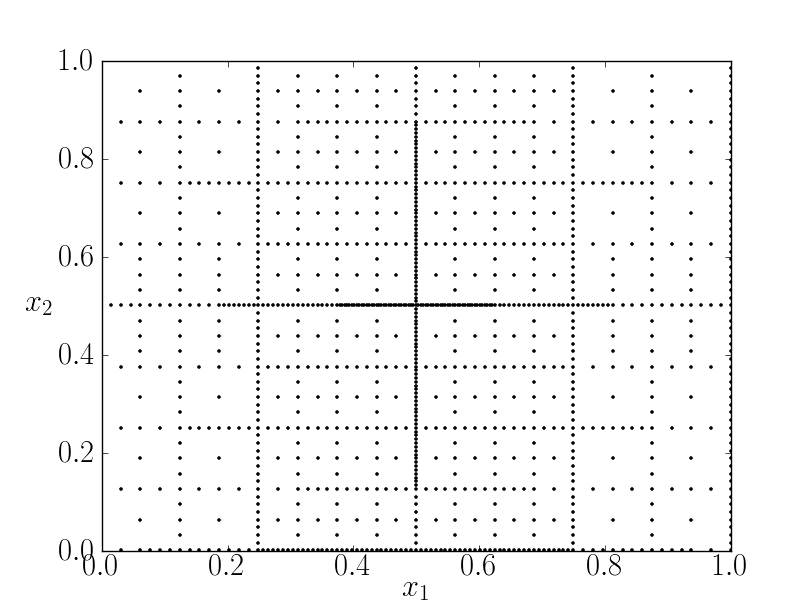}}
\end{subfigmatrix}
\caption{The adaptive grids obtained using basis of varying degree. Here the error criterion~\eqref{eq:absolute-error-criterion} is used with $\varepsilon=10^{-6}$.}
\label{fig:genz-f4-d2}
\end{figure}

\begin{figure}[htp]
\centering
\begin{subfigmatrix}{2}
\subfigure[$p=1$, ($9127$ points, $\varepsilon_{\ell^2}=3.19\cdot 10^{-3}$)]{\includegraphics[width=0.49\textwidth]{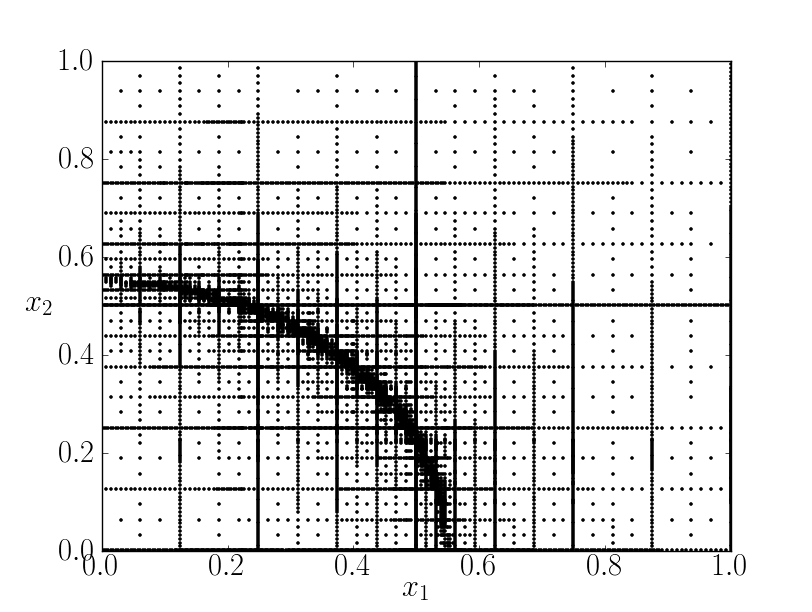}}
\subfigure[$p=2$, ($N=3980$, $\varepsilon_{\ell^2}=1.15\cdot10^{-2}$)]{\includegraphics[width=0.49\textwidth]{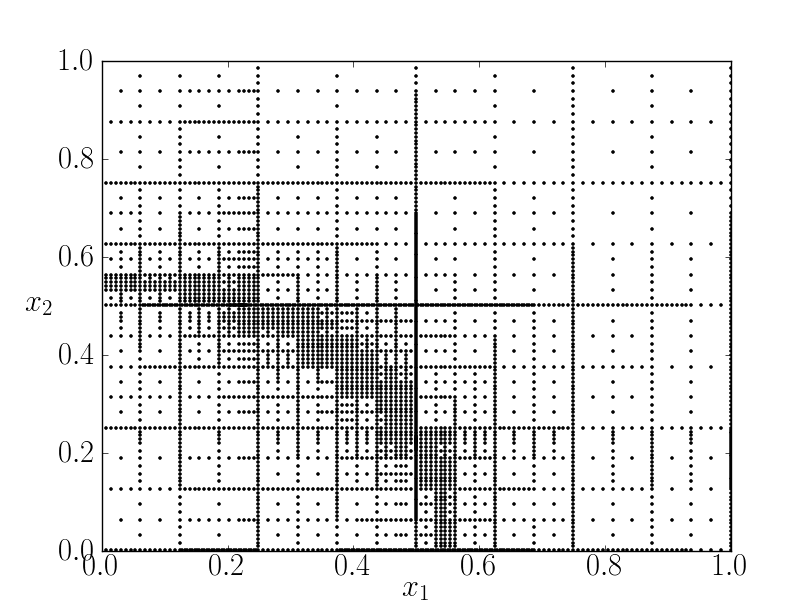}}
\end{subfigmatrix}
\caption{The adaptive grids obtained using basis of varying degree. Here the error criterion~\eqref{eq:absolute-error-criterion} is used with $\varepsilon=10^{-6}$.}
\label{fig:lineSing}
\end{figure}

\subsection{Increasing the Degree of the Local Polynomial Basis}
Let us now consider some moderate-dimensional integrals ($d=10$) and discuss the effect of the degree of the local polynomial
basis on the efficiency of the proposed method. Setting
$c_i=1/2^{i+2}$,
Figure~\ref{fig:degree-comparison}
compares the rates of convergence with respect to the tolerance $\varepsilon$ when the proposed method is applied to the 
three test functions~\eqref{eq:f3}-\eqref{eq:f5}.
Figure~\ref{fig:degree-comparison}~(a) illustrates convergence with respect to the 
$\varepsilon_{\ell^\infty}$ measure for the smooth function $f_2^{d=10}$. 
In this case the effect of the higher-degree basis is clearly
evident. The quadratic basis provides drastic improvement over the standard piecewise-linear basis and the quartic 
basis provides a further increase in efficiency.  

This result is mirrored when $h$-GSG is applied to the piecewise-continuous
function $f_3^{d=10}$~\eqref{eq:f4} (Figure~\ref{fig:degree-comparison}~(b)). However when a jump discontinuity is present (function~\eqref{eq:f5} $f_4^{d=10}$) the 
performance of the quartic basis functions is reduced (Figure~\ref{fig:degree-comparison}~(c)).

The quadratic basis significantly increases the accuracy of the $h$-GSG method for smooth and discontinuous functions.
Higher order basis functions $p>2$ provide further increases in the rate of convergence obtained for smooth problems,
but performance is degraded for discontinuous problems. These results are reproduced when higher-dimensional realizations of these
functions are considered.

\begin{figure}[htp]
\centering
\begin{subfigmatrix}{3}
\subfigure[$\varepsilon_{\ell^\infty}$ error in $f_2$]{\includegraphics[width=0.49\textwidth]{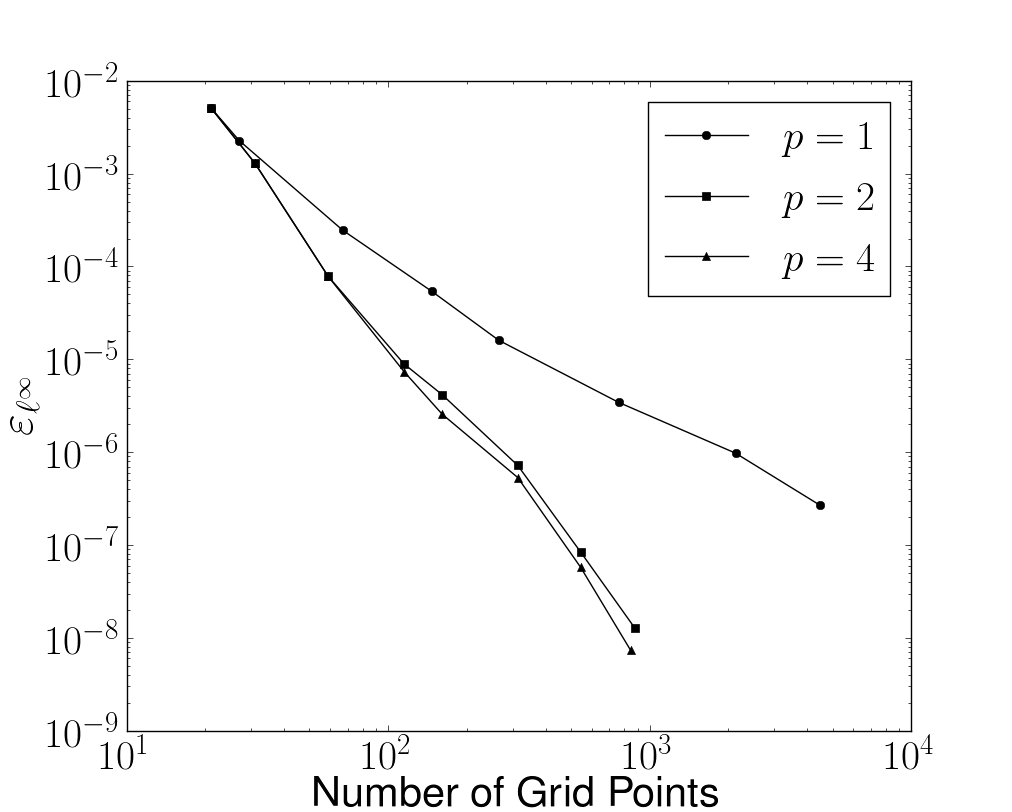}}
\subfigure[$\varepsilon_{\ell^\infty}$ error in $f_3$]{\includegraphics[width=0.49\textwidth]{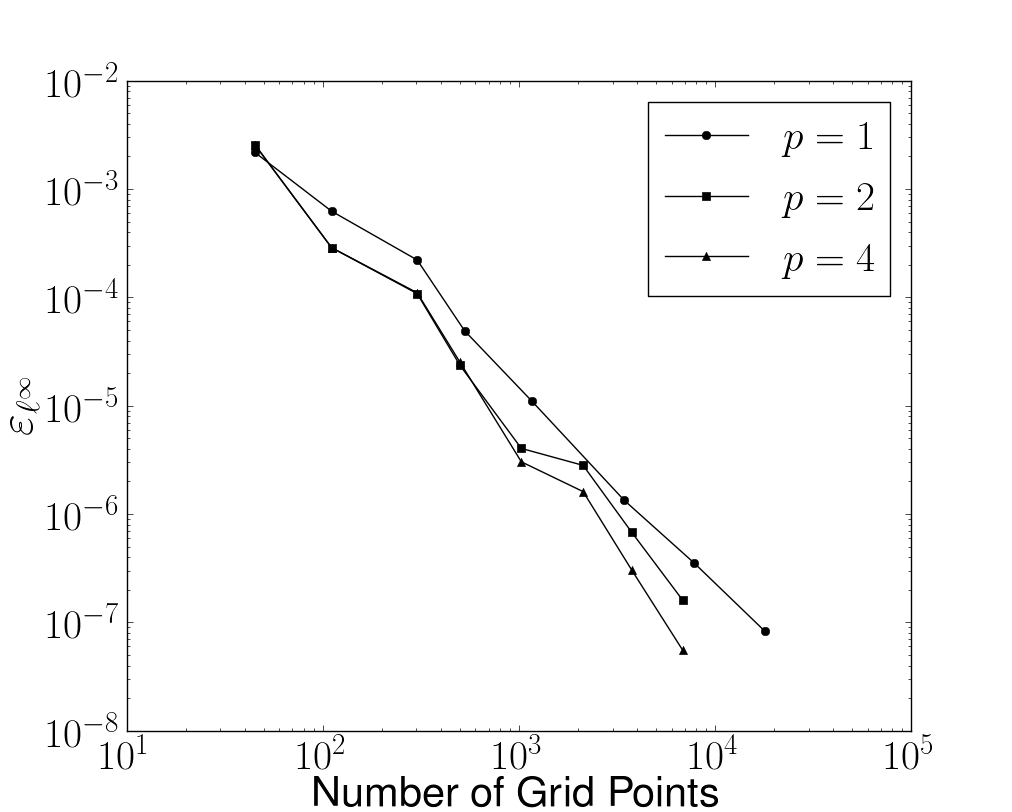}}
\subfigure[$\varepsilon_{\ell^2}$ error in $f_4$]{\includegraphics[width=0.49\textwidth]{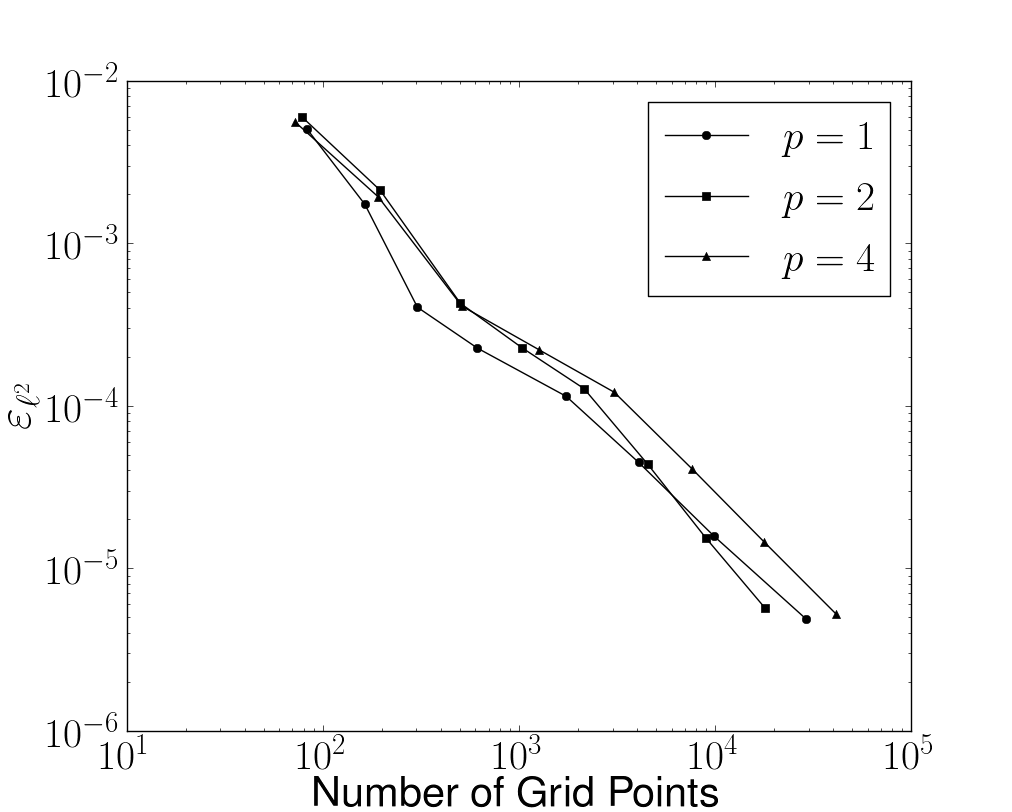}}
\end{subfigmatrix}
\caption{Error in the interpolants of $f_2$, $f_3$, and $f_4$ for $p\in\{1,2,4\}$, $\varepsilon=10^{-8}$ and $d=10$.}
\label{fig:degree-comparison}
\end{figure}

\subsection{Efficient Termination of $h$-GSG}
\label{sec:termination-comparison}
In Section~\ref{sec:efficientTermination} we proposed that the efficiency of the generalised sparse grid algorithm
can be improved by only adding admissible indices with $r_\bi>\varepsilon$ to the active index set. Here
we substantiate that claim. The $h$-GSG method discussed here and throughout this manuscript implements this 
modification.

Again consider the three test functions~\eqref{eq:f3}-\eqref{eq:f5}. But now let us investigate performance
when $d=100$ and 
\[
 c_i=\lambda\exp(-\frac{35\cdot i}{d}),\quad i=1,\ldots,d 
\]
where the parameter $\lambda$ which controls the effective dimensionality of the function.
Figures~\ref{fig:termination-condition-comparison} illustrates the difference between the
proposed method with and without the modification proposed in Section~\ref{sec:efficientTermination}.
Specifically the figures depict the error in the sparse grid interpolant as the algorithm evolves. The modification
results in substantial improvement when applied to the smooth (not depicted) and piecewise-continuous 
(Figure~\ref{fig:termination-condition-comparison}~(a)) functions. The unmodified algorithm
adds many points corresponding to grid indices with $r_\bi<\varepsilon$ and which contribute to the interpolation error.
The modification limits the number of unimportant points. 

When applied to the discontinuous function (Figure~\ref{fig:termination-condition-comparison}~(b))
the modification results in lower accuracy than when the unmodified algorithm is used. The unmodified algorithm continues to add points
belonging to grids with $r_\bi<\varepsilon$ and which contribute little to
the integral of the function yet still significantly influence the accuracy of the interpolant. The points belonging
to $r_\bi<\varepsilon$ mainly reside around the discontinuity. As the level of refinement increases, the contribution
of these points to the integral decreases yet their effect on the interpolant may not. This effect is illustrated in 
Figure~\ref{fig:termination-condition-comparison}~(c) which depicts the decrease in the error of the integral approximation with and without the
modification when $h$-GSG is applied to the discontinuous function. Here it is clear that adding grids with $r_\bi<\varepsilon$ has little effect on the accuracy of the 
integral approximation.
Also note that the effect of the termination condition decreases when the dimensionality $d$ is small.

\begin{figure}[htp]
\centering
\begin{subfigmatrix}{3}
\subfigure[$\varepsilon_{\ell^2}$ error in $f_3^{d=100}$]{\includegraphics[width=0.49\textwidth]{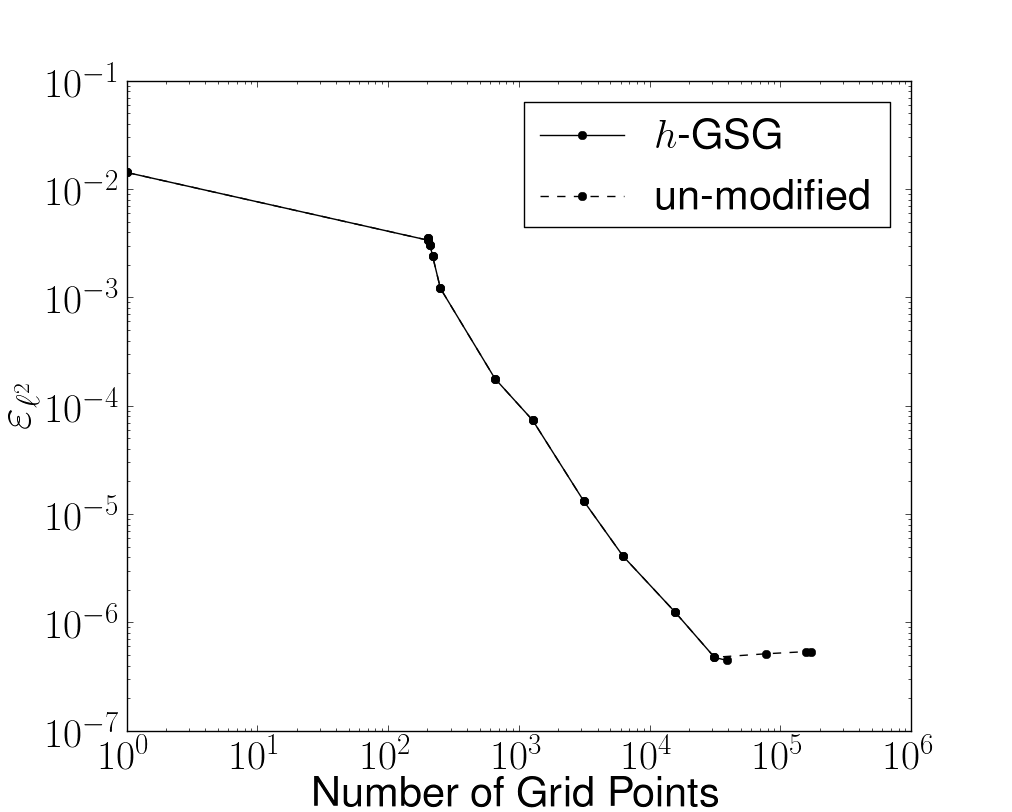}}
\subfigure[$\varepsilon_{\ell^2}$ error in $f_4^{d=100}$]{\includegraphics[width=0.49\textwidth]{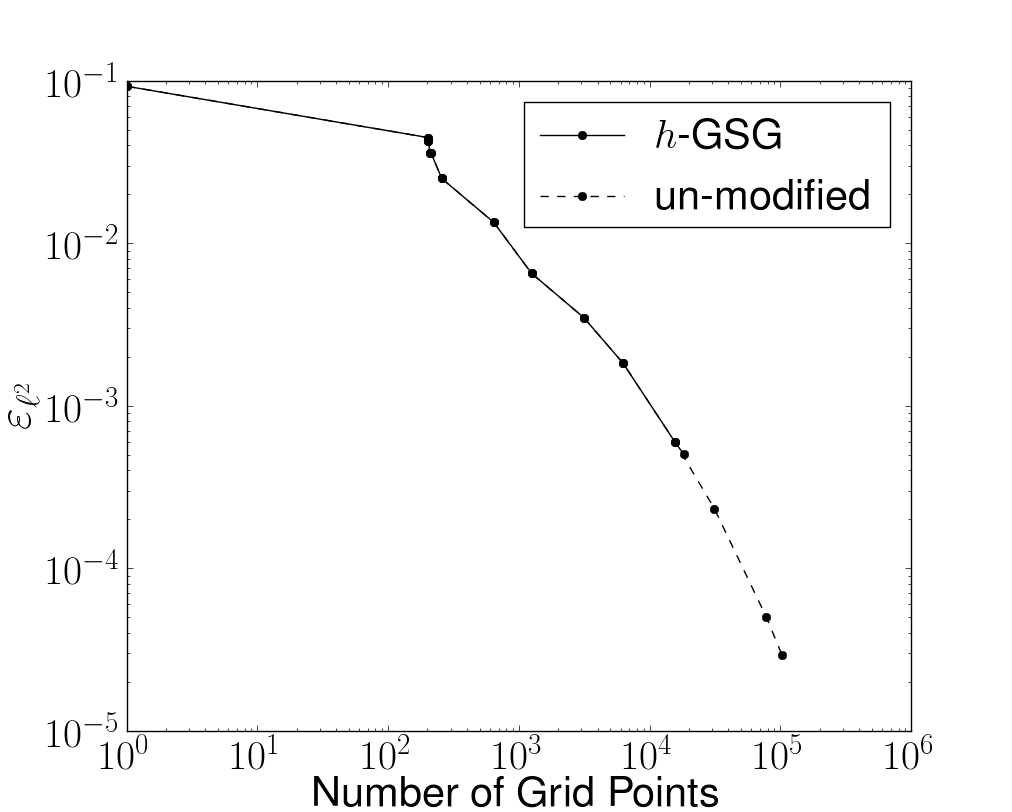}}
\subfigure[$\varepsilon_{\textrm{integral}}$ error in $f_4^{d=100}$]{\includegraphics[width=0.49\textwidth]{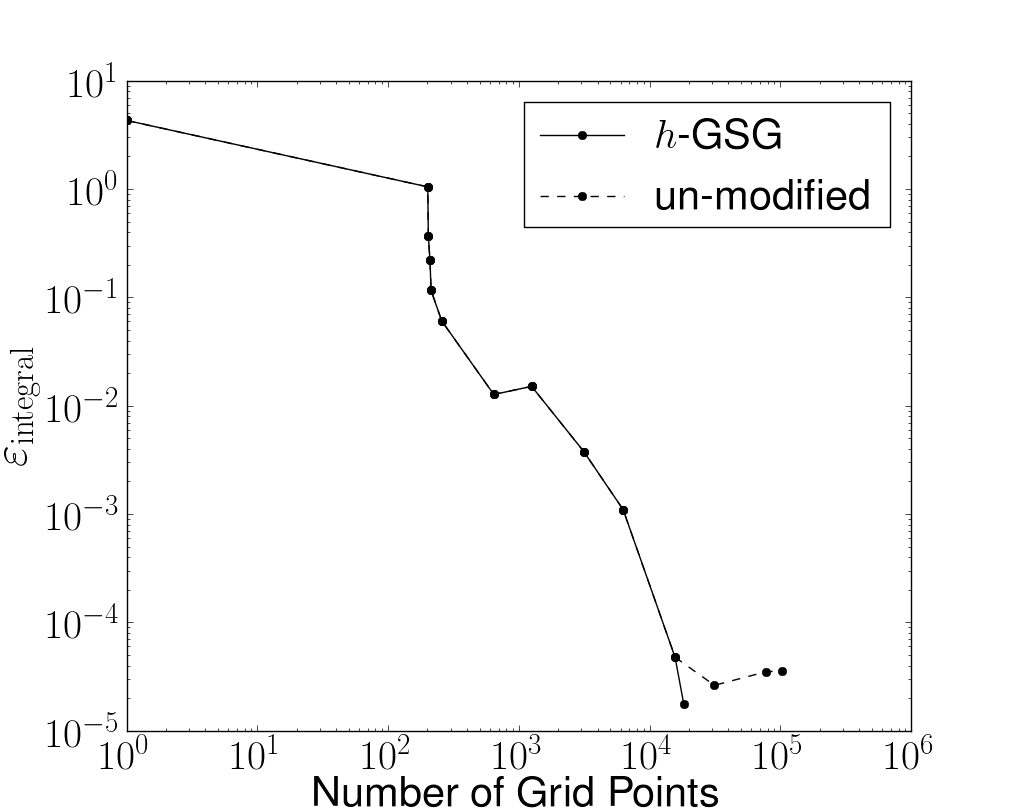}}
\end{subfigmatrix}
\caption{The evolution of the error in the interpolant when applied to the test functions with and without the proposed termination condition.
	Here $d=100$ and $\varepsilon=10^{-6}$. Initially the results are visually identical. Differences  only occur towards the end of the algorithm.}
\label{fig:termination-condition-comparison}
\end{figure}

\subsection{High-Dimensional Interpolation}
In Section~\ref{sec:termination-comparison} we used $h$-GSG for interpolation for a set of 100-dimensional problems. In this 
section we show that the proposed method can be applied to much higher-dimensional problems. Again consider the 
discontinuous function~\eqref{eq:f5} which is the most difficult function to approximate of the three test functions
used thus far. Again let
\begin{equation}
\label{eq:eigenDecay}
 c_i=\lambda\exp(-\frac{35\cdot i}{d})
\end{equation}


Table~\ref{tab:increaseDimHGSG} shows the number of function evaluations required to approximate $f_4$
and the resulting relative error in the approximated integral when $\varepsilon=10^{-4}$, quadratic basis functions are used,
and $\lambda=1$. An analytical expression for the integrand can be obtained easily due to the exponential nature of the function.
Arbitrary precision arithmetic was used to evaluate the numerical value of the reference integrals.
Due to the large range of values that $f_4$ can
take in high dimensions we use the relative error indicators
\[
  r_\bi=\left\lvert\frac{\sum_{\bj\in B_\bi} v_{\bi,\bj}\cdot w_{\bi,\bj}}{w_{\mathbf{0},\mathbf{0}}\cdot v_{\mathbf{0},\mathbf{0}}}\right\rvert,
 \quad \gamma_\bi=\left\lvert\frac{v_{\bi,\bj}\cdot w_{\bi,\bj}}{w_{\mathbf{0},\mathbf{0}}\cdot v_{\mathbf{0},\mathbf{0}}}\right\rvert
\]
to respectively guide difference space selection and local refinement. The $h$-GSG algorithm is terminated when 
\[\left\lvert \frac{\sum_{\bi\in\cI}r_\bi}{w_{\mathbf{0},\mathbf{0}}\cdot v_{\mathbf{0},\mathbf{0}}}\right\rvert<\varepsilon\]
where $r$ is the global error indicator used in Algorithm~\ref{alg:generalisedSG}.
An error of the order $10^{-2}$ is achieved for up to 700 dimensions using less than 300,000
 function evaluations. 

The accuracy of the integral approximation decays with increasing dimensionality. This 
is likely caused by the particular error indicators ($\gamma_{\bi,\bj}$ and $r_\bi$) used to guide adaptivity. 
At the moment a point $\bxi_{\bi,\bj}$
is refined if $\gamma_{\bi,\bj}\ge\varepsilon$ and a grid index $\bi$ is flagged for refinement only
if $r_\bi\ge\varepsilon$. This approach works well when $d<400$ but could be improved upon when the dimensionality
is higher. By excluding points if they have $\gamma_{\bi,\bj}$ less than the desired
accuracy $\varepsilon$ we are potentially ignoring a significant number of points whose combined contribution
to the integral is greater than $\varepsilon$. As the dimensionality increases more and more points will be excluded from
consideration thereby causing the accuracy of the approximant to decrease. 
This remark is consistent with Theorem~\ref{thm:optErr} which states that the accuracy of the 
$h$-GSG approximation is dependent on the number of points with $\gamma_{\bi,\bj}$ close to $\varepsilon$. 
As the number of these points increases the accuracy of the approximation decreases. 

The decrease in accuracy 
depicted in Table~\ref{tab:increaseDimHGSG} could be addressed by utilising more appropriate error criteria
 than those used here. 
The construction of efficient and robust error indicators is problem dependent 
and must be based upon the properties of the function under consideration. If no information on the function 
is available, we have shown that the error indicators used here will still perform well.

\begin{table}[htp]
\begin{center}
\caption{Errors in the $h$-GSG approximation of \eqref{eq:f5} for $d=100$ to $700$. $\varepsilon=10^{-5}$}
\label{tab:increaseDimHGSG}
\begin{tabular}{|c|c|c|}
\hline
$d$ & $N$ & $\varepsilon_\mathrm{integral}$\\
\hline\hline
100 & 3,376 & $3.81\cdot 10^{-4}$ \\ 
200 & 12,488 & $1.67\cdot 10^{-3}$ \\ 
300 & 31,533 & $1.71\cdot 10^{-4}$ \\ 
400 & 62,404 & $8.44\cdot 10^{-5}$ \\ 
500 & 109,356 & $4.57\cdot 10^{-3}$ \\
600 & 176,842 & $7.97\cdot 10^{-3}$ \\
700 & 269,665 & $1.68\cdot 10^{-2}$ \\
\hline
\end{tabular}
\end{center}
\end{table}

Here we note that the rate of convergence of the $h$-GSG method is governed by the implicit weighting of the importance
of each dimension. In this case the importance is controlled by the coefficients $c_i$. To illustrate this
dependence, Table~\ref{tab:increaseDimWeights} shows the efficiency of $h$-GSG as the magnitude of the coefficients
is increased. As the ``importance'' of each dimension increases the number of sub-dimensional components increases. In this
case an increase in $c_i$ also increases the function variability which also requires additional points to achieve a
set level of accuracy. 

\begin{table}[htp]
\begin{center}
\caption{Errors in the $h$-GSG approximation of \eqref{eq:f5} for increasing dimension 
importance $\varepsilon=10^{-6}$. Importance is increased by increasing $\lambda$ in Equation~\eqref{eq:eigenDecay}.}
\label{tab:increaseDimWeights}
\begin{tabular}{|c|c|c|}
\hline
$\lambda$ & $N$ & $\varepsilon_\mathrm{integral}$ \\
\hline\hline
1 & 9,226 & $1.66\cdot 10^{-4}$\\
2.5 & 34,977 & $2.96\cdot 10^{-5}$\\
5 & 175,201 & $6.53\cdot 10^{-4}$\\
7.5 & 659,368 & $1.93\cdot 10^{-3}$\\
\hline
\end{tabular}
\end{center}
\end{table}
\section{Conclusion}
This paper presented an $h$-adaptive generalised sparse grid ($h$-GSG) method 
for interpolating high-dimensional functions with discontinuities. The proposed algorithm
extends and improves upon existing approaches by combining the strengths of the generalised sparse grid
algorithm and hierarchical surplus-guided $h$-adaptivity.

The underlying generalised sparse grid algorithm greedily selects the subspaces that contribute most to 
the variability of a function. The hierarchical surplus of the points within each subspace is used as an error criterion
for $h$-refinement with the aim of concentrating computational effort within rapidly varying or discontinuous regions. 
This approach limits the number of points that are invested in `unimportant'
subspaces and regions within the high-dimensional domain.

A high-degree basis is used to obtain a high-order method that, given sufficient smoothness, performs 
significantly better than the traditional piecewise-linear basis. When discontinuities are present in the function
surface or its derivatives, performance deteriorates. However, it was shown numerically that even in such situations
the quadratic basis will still result in higher-rates of convergence than that achieved by using piecewise-linear
interpolation.

Often the importance of function variables are governed by natural yet unknown weights. 
In these cases, the proposed method can utilise this implicit weighting to determine and
restrict effort to the effective dimension of the model. This property allows the $h$-GSG method to be applied
to non-smooth functions with hundreds of variables. 

\bibliographystyle{spmpsci}

\bibliography{Collocation,SparseGrid,Misc,GPC,SensitivityAnalysis,EpistemicUncertainty,HDMR,JDJakeman,DiscontinuityDetection,Cubature}
\end{document}